\font \sevenrm=cmr7
\font \fiverm=cmr5
\documentclass[12pt, english]{amsart}
\input epsf.sty
\usepackage{amsfonts}
\usepackage{amssymb}
\usepackage{amsmath}
\usepackage{amsthm}
\usepackage{amscd}
\usepackage[all]{xy} 
\usepackage{epsfig,exscale}
\usepackage{color}
\usepackage{graphicx}
\usepackage{xspace}
\usepackage{axodraw4j}
\usepackage{colordvi}

\newcommand{\nc}{\newcommand}

{\everymath{\displaystyle\everymath{}}\array}%
{\endarray}
\newtheorem{theorem}{Theorem}
\newtheorem{definition}{Definition}

\newtheorem{proposition}{Proposition}
\newtheorem{ex}{Example}
\newtheorem{remark}{Remark}
\nc{\comment}[1]{[[{\tt #1}]] }
\nc{\Cal}[1]{{\mathcal {#1}}}
\nc{\mop}[1]{\mathop{\hbox {\rm #1} }\nolimits}
\nc{\gmop}[1]{\mathop{\hbox {\bf #1} }\nolimits}

\def\starg{{\displaystyle\mathop{\star}\limits}_g}

\nc{\smop}[1]{\mathop{\hbox {\sevenrm #1} }\nolimits}
\nc{\ssmop}[1]{\mathop{\hbox {\fiverm #1} }\nolimits}
\nc{\mopl}[1]{\mathop{\hbox {\rm #1} }\limits}
\nc{\smopl}[1]{\mathop{\hbox {\sevenrm #1} }\limits}
\nc{\ssmopl}[1]{\mathop{\hbox {\fiverm #1} }\limits}
\nc{\frakg}{{\frak g}}
\nc{\g}[1]{{\frak {#1}}}
\def \restr#1{\mathstrut_{\textstyle |}\raise-6pt\hbox{$\scriptstyle #1$}}
\def \srestr#1{\mathstrut_{\scriptstyle |}\hbox to
-1.5pt{}\raise-4pt\hbox{$\scriptscriptstyle #1$}}
\nc{\wt}{\widetilde} \nc{\wh}{\widehat}
\nc{\redtext}[1]{\textcolor{red}{#1}}
\nc{\bluetext}[1]{\textcolor{blue}{#1}}
\nc\fleche[1]{\mathop{\hbox to #1 mm{\rightarrowfill}}\limits}
\nc{\ignore}[1]{}
\def\semi{\mathrel{\times}\kern -.85pt\joinrel\mathrel{\raise
1.4pt\hbox{${\scriptscriptstyle |}$}}}
\nc\R{{\mathbb R}}
\nc\N{{\mathbb N}}
\nc\inver{^{-1}}
\nc\point{\hbox{\bf .}}
\nc\un{\hbox{\bf 1}}

\def\pre3{\,{\scalebox{0.20}{  
\begin{picture}(226,162) (319,-80)
    \SetWidth{3.0}
    \SetColor{Black}
    \Arc(432,-13.962)(57.658,123.711,416.289)
    \Line(320,-14)(374,-15)
    \Line(490,-14)(544,-14)
    \Arc(432,34)(32,180,540)
    \Line(425.001,62.001)(434.999,69.999)\Line(426.001,70.999)(433.999,61.001)
    \Line(430.001,-5.999)(437.999,7.999)\Line(427.001,4.999)(440.999,-2.999)
    \Line(435.001,-78.999)(444.999,-65.001)\Line(433.001,-67.001)(446.999,-76.999)
    \Text(440,75)[lb]{\Huge{\Black{$0$}}}
    \Text(445,-15)[lb]{\Huge{\Black{$1$}}}
    \Text(445,-91)[lb]{\Huge{\Black{$0$}}}
  \end{picture}}}\,}

\def\racine{{\scalebox{0.3}{ 
\begin{picture}(12,12)(38,-38)
\SetWidth{0.5} \SetColor{Black} \Vertex(45,-33){5.66}
\end{picture}}}}

\def\arbreca{\,{\scalebox{0.15}{
\begin{picture}(8,156) (370,-147)
\SetWidth{2}
\SetColor{Black}
\Line(374,-143)(374,-99)
\Vertex(374,-96){9}
\Vertex(375,-144){12}
\Line(374,-92)(374,-48)
\Vertex(374,-45){9}
\Line(374,-42)(374,2)
\Vertex(374,5){9}
\end{picture}
}}\,}

\def\arbrecc{\,{\scalebox{0.15}{
\begin{picture}(48,98) (349,-205)
\SetWidth{2}
\SetColor{Black}
\Vertex(375,-202){12}
\Line(376,-200)(395,-165)
\Line(373,-201)(354,-164)
\Vertex(353,-161){9}
\Vertex(395,-163){9}
\Line(353,-160)(353,-113)
\Vertex(353,-111){9}
\end{picture}
}}\,}
\def\arbrecd{\,{\scalebox{0.15}{
\begin{picture}(48,52) (349,-251)
\SetWidth{2}
\SetColor{Black}
\Vertex(375,-248){12}
\Line(376,-246)(395,-211)
\Line(373,-247)(354,-210)
\Vertex(353,-207){9}
\Vertex(395,-209){9}
\Line(375,-247)(375,-206)
\Vertex(376,-203){9}
\end{picture}
}}\,}
\def\arbreda{\,{\scalebox{0.15}{
\begin{picture}(8,204) (370,-99)
\SetWidth{2}
\SetColor{Black}
\Line(374,-95)(374,-51)
\Vertex(374,-48){9}
\Vertex(375,-96){12}
\Line(374,-44)(374,0)
\Vertex(374,3){9}
\Line(374,6)(374,50)
\Vertex(374,53){9}
\Line(374,53)(374,98)
\Vertex(374,101){9}
\end{picture}
}}\,}

\def\arbredc{\,{\scalebox{0.15}{
\begin{picture}(48,150) (349,-205)
\SetWidth{2}
\SetColor{Black}
\Line(376,-148)(395,-113)
\Line(373,-149)(354,-112)
\Vertex(353,-109){9}
\Vertex(395,-111){9}
\Line(353,-108)(353,-61)
\Vertex(353,-59){9}
\Line(374,-200)(374,-153)
\Vertex(374,-149){9}
\Vertex(374,-202){12}
\end{picture}
}}\,}

\def\arbrede{\,{\scalebox{0.15}{
\begin{picture}(48,153) (349,-150)
\SetWidth{2}
\SetColor{Black}
\Vertex(375,-147){12}
\Line(376,-145)(395,-110)
\Line(373,-146)(354,-109)
\Vertex(353,-106){9}
\Vertex(395,-108){9}
\Line(353,-105)(353,-58)
\Vertex(353,-56){9}
\Line(353,-52)(353,-5)
\Vertex(353,-1){9}
\end{picture}
}}\,}
\def\arbredf{\,{\scalebox{0.15}{
\begin{picture}(48,98) (349,-205)
\SetWidth{2}
\SetColor{Black}
\Vertex(375,-202){12}
\Line(376,-200)(395,-165)
\Line(373,-201)(354,-164)
\Vertex(353,-161){9}
\Vertex(395,-163){9}
\Line(353,-160)(353,-113)
\Vertex(353,-111){9}
\Line(395,-159)(395,-112)
\Vertex(395,-111){9}
\end{picture}
}}\,}

\def\arbredg{\,{\scalebox{0.15}{
\begin{picture}(48,98) (349,-205)
\SetWidth{2}
\SetColor{Black}
\Vertex(375,-202){12}
\Line(376,-200)(395,-165)
\Line(373,-201)(354,-164)
\Vertex(353,-161){9}
\Vertex(395,-163){9}
\Line(375,-201)(375,-160)
\Vertex(376,-157){9}
\Vertex(376,-111){9}
\Line(375,-155)(375,-114)
\end{picture}
}}\,}

\def\racine{\,{\scalebox{0.07}{
\begin{picture}(29,29) (360,-285)
    \SetWidth{6}
    \SetColor{Black}
    \Vertex(375,-271){20}
  \end{picture}
  }}\,}
\def\echela{\,{\scalebox{0.07}{
\begin{picture}(33,116) (353,-443)
    \SetWidth{6}
    \SetColor{Black}
    \Vertex(369,-428){20}
    \Vertex(369,-341){16}
    \Line(369,-341)(369,-419)
  \end{picture}
  }}\,}
  \def\echelb{\,{\scalebox{0.07}{
   \begin{picture}(33,195) (351,-363)
    \SetWidth{6}
    \SetColor{Black}
    \Vertex(369,-262){16}
    \Line(369,-262)(369,-340)
    \Vertex(368,-348){20}
    \Line(369,-184)(369,-262)
    \Vertex(369,-182){16}
  \end{picture}
  }}\,}
  \def\arbrey{\,{\scalebox{0.07}{
  \begin{picture}(147,114) (299,-444)
    \SetWidth{6}
    \SetColor{Black}
    \Vertex(368,-429){20}
    \Vertex(313,-344){16}
    \Vertex(434,-348){16}
    \Line(312,-344)(368,-428)
    \Line(433,-348)(372,-428)
  \end{picture}
}}\,}
\def\arbreza{\,{\scalebox{0.07}{
\begin{picture}(147,200) (299,-358)
    \SetWidth{6}
    \SetColor{Black}
    \Vertex(368,-343){20}
    \Vertex(313,-258){16}
    \Vertex(434,-262){16}
    \Line(312,-258)(368,-342)
    \Line(433,-262)(372,-342)
    \Vertex(313,-172){16}
    \Line(311,-179)(311,-252)
  \end{picture}
}}\,}
  \def\arbrema{\,{\scalebox{0.07}{
   \begin{picture}(149,215) (296,-345)
    \SetWidth{6}
    \SetColor{Black}
    \Line(370,-244)(370,-317)
    \Vertex(371,-238){16}
    \Line(310,-149)(366,-233)
    \Line(432,-150)(372,-234)
    \Vertex(310,-144){16}
    \Vertex(433,-144){16}
    \Vertex(371,-330){20}
  \end{picture}
  }}\,}
  \def\arbremg{\,{\scalebox{0.07}{
  \begin{picture}(244,379) (201,-181)
    \SetWidth{6}
    \SetColor{Black}
    \Line(370,-80)(370,-153)
    \Vertex(371,-74){16}
    \Line(310,15)(366,-69)
    \Line(432,14)(372,-70)
    \Vertex(310,20){16}
    \Vertex(433,20){16}
    \Vertex(371,-166){20}
    \Vertex(264,98){16}
    \Vertex(359,100){16}
    \Vertex(215,183){16}
    \Vertex(302,184){16}
    \Line(359,99)(315,26)
    \Line(299,175)(265,103)
    \Line(214,182)(259,105)
    \Line(267,91)(308,25)
  \end{picture}
  }}\,}
   \def\arbremi{\,{\scalebox{0.07}{
  \begin{picture}(205,388) (240,-172)
    \SetWidth{6}
    \SetColor{Black}
    \Line(370,-71)(370,-144)
    \Vertex(371,-65){16}
    \Line(310,24)(366,-60)
    \Line(432,23)(372,-61)
    \Vertex(310,29){16}
    \Vertex(433,29){16}
    \Vertex(371,-157){20}
    \Line(310,108)(310,35)
    \Vertex(309,118){16}
    \Line(374,207)(314,123)
    \Vertex(369,202){16}
    \Line(250,207)(306,123)
    \Vertex(254,202){16}
    \Line(434,109)(434,36)
    \Vertex(433,115){16}
  \end{picture}
  }}\,}
   \def\arbremj{\,{\scalebox{0.07}{
  \begin{picture}(235,388) (240,-172)
    \SetWidth{6}
    \SetColor{Black}
    \Line(370,-71)(370,-144)
    \Vertex(371,-65){16}
    \Line(310,24)(366,-60)
    \Line(432,23)(372,-61)
    \Vertex(310,29){16}
    \Vertex(433,29){16}
    \Vertex(371,-157){20}
    \Line(310,108)(310,35)
    \Vertex(309,118){16}
    \Line(374,207)(314,123)
    \Vertex(369,202){16}
    \Line(250,207)(306,123)
    \Vertex(254,202){16}
    \Vertex(463,-68){16}
    \Line(459,-73)(379,-153)
  \end{picture}
  }}\,}
\def\arbreml{\,{\scalebox{0.07}{
\begin{picture}(205,476) (240,-84)
    \SetWidth{6}
    \SetColor{Black}
    \Line(370,17)(370,-56)
    \Vertex(371,23){16}
    \Vertex(371,-69){20}
    \Line(312,286)(312,213)
    \Vertex(312,292){16}
    \Line(256,375)(312,291)
    \Vertex(254,378){16}
    \Line(377,379)(317,295)
    \Vertex(375,376){16}
    \Line(434,111)(374,27)
    \Line(312,107)(368,23)
    \Vertex(433,113){16}
    \Vertex(311,111){16}
    \Line(312,197)(312,124)
    \Vertex(312,202){16}
  \end{picture}
  }}\,}
\def\arbremo{\,{\scalebox{0.07}{
 \begin{picture}(205,307) (243,-253)
    \SetWidth{6}
    \SetColor{Black}
    \Line(370,-152)(370,-225)
    \Vertex(371,-146){16}
    \Line(310,-57)(366,-141)
    \Line(432,-58)(372,-142)
    \Vertex(310,-52){16}
    \Vertex(433,-52){16}
    \Vertex(371,-238){20}
    \Line(371,-64)(371,-137)
    \Vertex(371,-51){16}
    \Line(438,38)(378,-46)
    \Vertex(436,40){16}
    \Vertex(257,40){16}
    \Line(258,32)(308,-52)
    \Line(328,34)(370,-43)
    \Vertex(328,40){16}
  \end{picture}
   }}\,}
 \def\arbremr{\,{\scalebox{0.07}{
 \begin{picture}(154,306) (296,-254)
    \SetWidth{6}
    \SetColor{Black}
    \Line(370,-153)(370,-226)
    \Vertex(371,-147){16}
    \Line(310,-58)(366,-142)
    \Line(432,-59)(372,-143)
    \Vertex(310,-53){16}
    \Vertex(433,-53){16}
    \Vertex(371,-239){20}
    \Line(371,-65)(371,-138)
    \Vertex(371,-52){16}
    \Line(437,40)(377,-44)
    \Line(311,38)(367,-46)
    \Vertex(311,38){16}
    \Vertex(433,37){16}
    \Line(437,-149)(377,-233)
    \Vertex(438,-149){16}
  \end{picture}
  }}\,}
\def\arbremt{\,{\scalebox{0.07}{
\begin{picture}(275,286) (234,-274)
    \SetWidth{6}
    \SetColor{Black}
    \Vertex(371,-259){20}
    \Line(309,-169)(365,-253)
    \Line(439,-165)(379,-249)
    \Vertex(437,-165){16}
    \Vertex(306,-167){16}
    \Line(245,-76)(301,-160)
    \Line(502,-76)(442,-160)
    \Line(422,-87)(436,-162)
    \Vertex(321,-75){16}
    \Vertex(497,-81){16}
    \Vertex(248,-81){16}
    \Vertex(420,-75){16}
    \Line(319,-83)(306,-157)
    \Line(249,-6)(249,-74)
    \Vertex(250,-2){16}
  \end{picture}
}}\,}
   \def\arbremv{\,{\scalebox{0.07}{
  \begin{picture}(275,208) (234,-352)
    \SetWidth{6}
    \SetColor{Black}
    \Vertex(371,-337){20}
    \Line(309,-247)(365,-331)
    \Line(439,-243)(379,-327)
    \Vertex(437,-243){16}
    \Vertex(306,-245){16}
    \Line(245,-154)(301,-238)
    \Line(502,-154)(442,-238)
    \Line(422,-165)(436,-240)
    \Vertex(497,-159){16}
    \Vertex(248,-159){16}
    \Vertex(421,-158){16}
    \Line(371,-248)(371,-328)
    \Vertex(371,-241){16}
    \Line(338,-162)(307,-241)
    \Vertex(336,-159){16}
  \end{picture}
  }}\,}
   \def\arbremw{\,{\scalebox{0.07}{
  \begin{picture}(309,284) (200,-276)
    \SetWidth{6}
    \SetColor{Black}
    \Vertex(371,-261){20}
    \Line(309,-171)(365,-255)
    \Line(439,-167)(379,-251)
    \Vertex(437,-167){16}
    \Vertex(306,-169){16}
    \Line(245,-78)(301,-162)
    \Line(502,-78)(442,-162)
    \Line(422,-89)(436,-164)
    \Vertex(497,-83){16}
    \Vertex(248,-83){16}
    \Vertex(420,-77){16}
    \Vertex(214,-7){16}
    \Vertex(286,-6){16}
    \Line(216,-13)(244,-80)
    \Line(284,-12)(253,-79)
  \end{picture}
  }}\,}

  \def\arbreha{\,{\scalebox{0.07}{
   \begin{picture}(149,299) (296,-261)
    \SetWidth{6}
    \SetColor{Black}
    \Vertex(371,-154){16}
    \Line(310,-65)(366,-149)
    \Line(432,-66)(372,-150)
    \Vertex(310,-60){16}
    \Vertex(433,-60){16}
    \Vertex(371,-246){20}
    \Line(310,16)(310,-57)
    \Line(371,-163)(371,-236)
    \Vertex(311,24){16}
  \end{picture}
  }}\,}
  \def\arbrehd{\,{\scalebox{0.07}{
   \begin{picture}(158,215) (296,-345)
    \SetWidth{6}
    \SetColor{Black}
    \Vertex(371,-238){16}
    \Line(310,-149)(366,-233)
    \Line(432,-150)(372,-234)
    \Vertex(310,-144){16}
    \Vertex(433,-144){16}
    \Vertex(371,-330){20}
    \Line(371,-247)(371,-320)
    \Vertex(442,-247){16}
    \Line(440,-251)(376,-327)
  \end{picture}
  }}\,}

\def\diagramme #1{\vskip 4mm \centerline {#1} \vskip 4mm}
\begin{document}
\title{
{Doubling pre-Lie algebra of rooted trees}}

\author{Mohamed Belhaj Mohamed}
\address{{Mathematics Departement, Sciences college, Taibah University, Kingdom of Saudi Arabia~.}\vspace{0.01cm}
{Laboratoire de math\'ematiques physique fonctions sp\'eciales et applications, Universit\'e de Sousse, rue Lamine Abassi 4011 H. Sousse,  Tunisie.}}     
        \email{mohamed.belhajmohamed@isimg.tn}

\date{September 2019}
\noindent{\footnotesize{${}\phantom{a}$ }}
\begin{abstract}
We study the pre-Lie algebra of rooted trees $(\Cal T, \rightarrow)$ and we define a pre-Lie structure on its doubling space $(V, \leadsto)$. Also, we find the enveloping algebras of the two pre-Lie algebras denoted respectively by $(\Cal H', \star, \Gamma)$ and $(\Cal D', \bigstar, \chi)$. We prove that $(\Cal D', \bigstar, \chi)$ is a module-bialgebra on $(\Cal H', \star, \Gamma)$ and we find some relations between the two pre-Lie structures.
\end{abstract}
\maketitle

\tableofcontents
\section{Introduction}  
 Rooted trees appeared in the work of Cayley \cite{ca} in the context of differential equations, Butcher \cite{bu}, Grossman and Larson \cite{gl}, Munthe-Kaas and Wright \cite{mw} in the field of numerical analysis. They are used in the context of renormalization in perturbative quantum field theory in the works of A. Connes and D. Kreimer \cite{A.D2000, ad98, DK98}, D. Calaque, K. Ebrahimi-Fard, D. Manchon \cite{ckm, ms} and L. Foissy \cite{lf}.

In 1998, A. Connes and D. Kreimer \cite{ad98, DK98} showed that $\Cal{H} = S({\Cal T})$, where $\Cal T$ be the vector space spanned by  rooted trees,  admits a structure of graded Hopf algebra. The product is the concatenation $m$ , and the coproduct is defined by:
\begin{eqnarray*}
\Delta (t) &=&\sum_{c \in {\mop{\tiny{Adm}}(t)}} P^c(t)\otimes R^c(t),
\end{eqnarray*}
where $\mop{Adm}(t)$ is the set of admissible cuts of a forest. He used this structure to solve some problems related to the renormalization of quantum field theory.

Grafting pre-Lie algebra of rooted trees was  studied for the first time by F. Chapoton and M. Livernet \cite{cha} as being the space of primitive elements in the graded dual of the Hopf algebra of rooted trees. They used this structure to give a combinatorial description of the pre-Lie operad in terms of rooted trees.

In joint work with Dominique Manchon \cite{DB}, we have studied the doubling bialgebra in the context of rooted trees. We have defined the doubling bialgebras of rooted trees given by extraction contraction and admissible cuts, and we have shown the existence of many relations between these two structures.

In this article, we start by defining the enveloping algebra of grafting pre-Lie algebra of rooted trees using the method of Oudam and Guin \cite{og}. We consider the Hopf symmetric algebra $\Cal H' : = \Cal S (\Cal {T})$ of the pre-Lie algebra $(\Cal {T}, \rightarrow)$, equipped with its usual unshuffling coproduct $\Gamma$ and a product $\star$ defined on $\Cal H'$ by:
\begin{eqnarray*}
t \star t' &=& \sum_{(t)}  t^{(1)} (t^{(2)} \rightarrow t').
\end{eqnarray*}
We find then a comodule-coalgebra structure connecting this last Hopf algebra and the Hopf algebra of Connes-Kreimer.

Secondly, we define a pre-Lie structure $(V, \leadsto)$ on the doubling space of rooted trees. The product $\leadsto$ is defined by:
$$(t_1, s_1)  \leadsto (t_2, s_2) :=  \sum_{v \in \Cal V(t_2 - s_2)}  (t_1 \rightarrow_v t_2 , s_1  s_2),$$
and we construct $(\Cal D', \bigstar, \chi)$ the enveloping algebra of $(V, \leadsto)$, where $\chi$ is the usual unshuffling coproduct  and $\bigstar$ is defined by: 
\begin{eqnarray*}
(t, s) \bigstar (t', s') &=& \sum_{(t, s)}  (t, s)^{(1)} \big((t, s)^{(2)} \leadsto  (t', s')\big).
\end{eqnarray*}

We show that $(\Cal D', m, \chi)$ is a comodule-coalgebra on $(\Cal D,  m, \Delta)$, and in the last section we give some relations between the two structures.  We prove that $(V, \Diamond)$ is a left pre-Lie module on ${\Cal T}$, and we prove taht $(\Cal D', \bigstar, \chi)$ is a module-bialgebra on $(\Cal H', \star, \Gamma)$. 

The Connes-Kreimer Hopf algebra of rooted trees and the enveloping algebra of the pre-Lie algebra of rooted trees are two important examples of Hopf algebras, so it is necessary to find relationships that connect the two structures. Our present work, which relates these two Hopf algebras and their  doubling structures, will help us to find other combinatorial results as well as others related to renormalization in quantum field theory.

\vspace{1cm}
\noindent
{\bf Acknowledgements:} I would like to thank Dominique Manchon for his support and his valuable comments.


\section{Bialgebra and doubling bialgebra of rooted trees}
A {\sl rooted tree\/} is a finite connected simply connected oriented graph such that every vertex has exactly one incoming edge, except for a distinguished vertex (the root) which has no incoming edge.
The set of rooted trees is denoted by $T$ and the set of rooted trees with $n$ vertices is denoted by $T_n$.   
\begin{ex}
\begin{eqnarray*}
T_1 &=& \{\racine\}\\
T_2 &=& \{\echela\}\\
T_3 &=& \{\echelb , \arbrey \}\\
T_4 &=& \{\arbreca, \arbrema, \arbreza, \arbrecd\}
\end{eqnarray*}
\end{ex}
Let $\Cal T$ the vector space spanned by the elements of $T$ and let $\Cal{H} = S({\Cal T})$ be the algebra of rooted trees. A. Connes and D. Kreimer \cite{ad98, DK98} showed that this space, graded according to the number of vertices, admits a structure of graded bialgebra. The product is the disjoint union, and the coproduct is defined by:
\begin{eqnarray*}
\Delta (t) &=& t \otimes \un + \un \otimes t + \sum_{c \in {\mop{\tiny{Adm'}}(t)}} P^c(t)\otimes R^c(t)\\
&=&\sum_{c \in {\mop{\tiny{Adm}}(t)}} P^c(t)\otimes R^c(t),
\end{eqnarray*}
where $\mop{Adm}(t)$ (resp $\mop{Adm'}(t)$) is the set of admissible cuts (resp. nontrivial admissible cuts) of a forest, i.e. the set of collections of edges such that any path from the root to a leaf contains at most one edge of the collection. We denote as usual by $P^c(t)$ (resp. $R^c(t)$) the pruning (resp. the trunk) of $t$, i.e. the subforest formed by the edges above the cut $c \in \mop{Adm}(t)$ (resp. the subforest formed by the edges under the cut). Note that the trunk of a tree is a tree, but the pruning of a tree may be a forest. We denote by $\un$ the empty forest, which is the unit. One sees easily that $\deg(t) = \deg(P_c(t)) + \deg(R_c(t))$ for all admissible cuts. (See \cite{ckm} and \cite{lf}).

\begin{ex}
$$
\Delta (\arbreza)=\un \otimes\arbreza+\arbreza\otimes
\un +\racine\otimes\echelb+\echela\otimes\echela+\racine\otimes\arbrey+\echela\racine\otimes\racine
+\racine\racine\otimes\echela.
$$
\end{ex}

\vspace{0.5cm}
Let ${V}$ the vector space spanned by the couples $(t,s)$ where $t$ is a tree and $s =P^{c_0}(t)$ where $c_0$ is an admissible cut of $t$. We define then the doubling bialgebra of trees ${\Cal D} : = S({V})$, the product is given by:
$$m\big((t, s) \otimes (t', s')\big) = (t, s)(t', s') = (tt', ss'),$$
the coproduct $\Delta$ is defined for all $(t,s) \in \Cal D$ by: 
$$\Delta(t,s) =  \sum_{c \in {\mop{\tiny{Adm}}(s)}} \big(t, P^c(s)\big)\otimes \big( R^c(t), R^c(s)\big),$$
the unit $\textbf{1}$ is identified to empty forest, the counit $\varepsilon$ is given by $\varepsilon (t, s) = \varepsilon (s)$ and the graduation is given by the number of vertices of $s$:
$$ |(t, s)|  = |s|.$$
\begin{theorem}\cite{DB}
$D$ is a graded bialgebra.
\end{theorem}
\begin{remark}
We remark here that $\Delta (V) \subset V \otimes V$. Indeed, if $(t, s) \in V$ then $\big(t,P^{c}(s)\big) \in V$, since a pruning of $s$ is also a pruning of $t$. Similary $\big( R^c(t), R^c(s)\big) \in V$ because $R^c(t)$ is a tree, and $R^c(s) = R^c(P^{c_0} (t)) = P^{c_0} (R^c(t))$ is a pruning of $R^c(t)$. So we can restrict the coassociative coproduct $\Delta$ to $V$.  
\end{remark}
\begin{proposition}\label{p1} The second projection 
\begin{eqnarray*}
P_2:  \Cal D &\longrightarrow& {{\Cal H}} \\ (t, s) &\longmapsto& s
\end{eqnarray*} 
is a bialgebra morphism.
\end{proposition}
\begin{proof}
The fact that $P_2$ is an algebra morphism is trivial, it suffices to show that $P_2$ is a coalgebra morphism, i.e. $P_2$ verifies the following commutative diagram: 
\diagramme{
\xymatrix{ \Cal D \ar[rrr]^{P_2}\ar[d]_{\Delta}
&&&\Cal H\ar[d]^{\Delta}\\
\Cal D \otimes \Cal D
\ar[rrr]_{P_2 \otimes P_2}&&&\Cal H \otimes \Cal H
}
}
which can be seen by direct calculation:
\begin{eqnarray*}
\Delta \circ P_2 (t, s) &=& \Delta(s) \\  
&=&\sum_{c \in {\mop{\tiny{Adm}}(s)}} P^c(s)\otimes  R^c(s)\\
&=&\sum_{c \in {\mop{\tiny{Adm}}(s)}} P_2 \big(t, P^c(s)\big)\otimes P_2 \big( R^c(t), R^c(s)\big)\\
&=& (P_2 \otimes P_2) \Delta (t, s).
\end{eqnarray*}
\end{proof}

\section{The enveloping algebra of pre-Lie algebra}
In this section, we describe the method of Oudom and Guin \cite{og} to find the enveloping algebra of a pre-Lie algebra.

\begin{definition}
A Lie algebra over a field $k$ is a vector space $V$ endowed with a bilinear bracket $[., .]$ satisfying:
\begin{enumerate}
\item the antisymmetry:
$$
[x , y] = - [y , x] \hspace{3cm} \forall x, y \in V.
$$
\item the Jacobi identity:
$$
[x , [y , z]] + [y , [z , x]] + [z , [x , y]]  = 0 \hspace{3cm} \forall x, y , z \in V.
$$
\end{enumerate}
\end{definition}

\begin{definition}\cite{ch, dm11}
A left pre-Lie algebra over a field $k$ is a $k$-vector space $\Cal A$ with a binary bilinear product 
$\triangleright$ that satisfies the left pre-Lie identity:
\begin{equation}
(x \triangleright y) \triangleright z - x \triangleright (y \triangleright z) = (y \triangleright x) \triangleright z -  y \triangleright (x \triangleright z), 
\end{equation}
for all $x$, $y$, $z \in \Cal A$. Analogously, a right pre-Lie algebra is a $k$-vector space $\Cal A$ with a
binary bilinear product $\triangleleft$ that satisfies the right pre-Lie identity:
\begin{equation}
(x \triangleleft y) \triangleleft z - x \triangleleft (y \triangleleft z) = (x \triangleleft z) \triangleleft y -  x \triangleleft (z \triangleleft y).
\end{equation}
\end{definition}

As any right pre-Lie algebra $(\Cal A, \triangleleft )$ is also a left pre-Lie algebra with product
$x \triangleright y := y \triangleleft x$, we will only consider left pre-Lie algebras for the moment. The left
pre-Lie identity rewrites as:
\begin{equation}
L_{[x,  y]} = [L_x,  L_y],
\end{equation}
where $L_x : A  \longrightarrow A$ is defined by $L_x y = a \triangleright b$, and where the bracket on the
left-hand side is defined by $[a, b] := a \triangleright b -  b  \triangleright a$. As a consequence this bracket
satisfies the Jacobi identity.

\begin{definition}\label{odam}\cite{og}
Let $(A, \triangleright)$ be a pre-Lie algebra. We consider the Hopf symmetric algebra $\Cal S (A)$ equipped with its usual coproduct $\Delta$. We extend the product $\triangleright$ to $\Cal S (A)$. Let $a, b$ and $c \in \Cal S (A)$, and $x \in A$. We put:
\begin{eqnarray*}
\un \triangleright a &=& a\\
a \triangleright \un &=& \varepsilon(a) \un\\
(x a) \triangleright b &=& x \triangleright (a \triangleright b) - (x \triangleright a) \triangleright b\\
a \triangleright (b  c) &=& \sum_{a}  (a^{(1)} \triangleright b) (a^{(2)} \triangleright c).
\end{eqnarray*}

On $\Cal S (A)$, we define a product $\star$ by:
\begin{eqnarray*}
a \star b &=& \sum_{a}  a^{(1)} (a^{(2)} \triangleright b).
\end{eqnarray*}
\end{definition}
\begin{theorem} The space $(\Cal S (A), \star, \Delta)$is a Hopf algebra which is isomorphic to the enveloping Hopf algebra $\Cal U(A_{Lie})$ of the Lie algebra $A_{Lie}$. 
\end{theorem}
\begin{proof} This theorem was proved by Oudom and Guin in \cite{og}. 
\end{proof}

\section{The enveloping algebra of the pre-Lie algebra of rooted trees}

Grafting pre-Lie algebra  of rooted trees were studied for the first time by F. Chapoton and M. Livernet \cite{cha}, D. Manchon and A. Saidi \cite{ms} and after that by others in different domains. The grafting product is given, for all $t, s \in \Cal T$, by:
$$t \rightarrow s = \sum_{v \in V(s)} t \rightarrow_v s, $$
where $t \rightarrow_v s$ is the tree obtained by grafting the root of $t$ on the vertex $v$ of $s$.
More explicitly, the operation $t \rightarrow s$ consists of grafting the root of $t$ on every vertex
of $s$.
\begin{ex}
\begin{eqnarray*}
\racine\rightarrow\echela &=& \echelb +  \arbrey\\
\arbrey\rightarrow\arbreha   &=& \arbreml+\arbremg+\arbremo+\arbremi+\arbremt\\
\arbrey\rightarrow\arbrehd  &=&\arbremj+\arbremr+2\arbremv+\arbremw .
\end{eqnarray*}
\end{ex}
\begin{theorem}\cite{Dbu, DL} Equipped by $\rightarrow$, the space  $\Cal T$ is a pre-Lie algebra. 
\end{theorem}
 
Now, we can use the method of Oudom and Guin \cite{og} to find the enveloping algebra of the grafting pre-Lie algebra of rooted trees. We consider the Hopf symmetric algebra $\Cal H' : = \Cal S (\Cal {T})$ of the pre-Lie algebra $(\Cal {T}, \rightarrow)$, equipped with its usual unshuffling coproduct $\Gamma$. We extend the product $\rightarrow$ to $\Cal H'$ by the same method used in Definition \ref{odam} and we define a product $\star$ on $\Cal H'$ by:
\begin{eqnarray*}
t \star t' &=& \sum_{(t)}  t^{(1)} (t^{(2)} \rightarrow t').
\end{eqnarray*}
By construction, the space $(\Cal H', \star, \Gamma)$ is a Hopf algebra. 
\begin{ex} 
\begin{eqnarray*}
(\echela\racine)\rightarrow\echela &=& \echela \rightarrow(\racine \rightarrow  \echela) - (\echela \rightarrow\racine) \rightarrow  \echela \\
&=& \echela \rightarrow (\arbrey + \echelb) - \echelb \rightarrow  \echela \\
&=& 2 \arbrede + \arbredg  + \arbredf + \arbredc + \arbreda - \arbrede - \arbreda\\
&=& \arbrede + \arbredg  + \arbredf + \arbredc.
\end{eqnarray*}
\begin{eqnarray*}
(\echela\racine)\star\echela &=& (\echela\racine)\rightarrow\echela + \echela\racine\echela+    \echela (\racine \rightarrow  \echela) + \racine (\echela \rightarrow\echela ) \\
&=& \arbrede + \arbredg  + \arbredf + \arbredc + \echela\racine\echela + \echela (\arbrey + \echelb) + \racine (\arbrecc + \arbreca)\\
&=& \arbrede + \arbredg  + \arbredf + \arbredc + \echela\racine\echela + \echela \arbrey + \echela \echelb + \racine \arbrecc + \racine \arbreca.
\end{eqnarray*}
\end{ex}
\begin{theorem} $(\Cal H', m, \Gamma)$ is a comodule-coalgebra on $(\Cal H,  m, \Delta)$.
\end{theorem}
\begin{proof} 
It is clear that $\Delta : \Cal H' \longrightarrow  \Cal H \otimes  \Cal H'$ is a coaction, that means that $\Delta$ is coassocitive. 

Second, we prove that the coproduct $\Gamma$ is morphism of left  $\Cal H'$-comodules. This amounts to the commutativity of the following diagram:
\diagramme{
\xymatrix{
 \Cal H' \ar[d]_{\Gamma} \ar[rr]^{\Delta} 
&&  {\Cal H} \otimes    {\Cal H'} \ar[d]^{I \otimes \Gamma}\\
    {\Cal H'}  \otimes   {\Cal H'}  \ar[d]_{\Delta\otimes \Delta}&& {\Cal H} \otimes   {\Cal H'} \otimes   {\Cal H'}  \\
   {\Cal H}  \otimes   {\Cal H'}  \otimes    {\Cal H}  \otimes    {\Cal H'} \ar[rr]_{\tau^{23}} 
&&\ar[u]_{m \otimes I }    {\Cal H} \otimes   {\Cal H}  \otimes    {\Cal H'}  \otimes   {\Cal H'} 
 }
}
We use the shorthand notation: $m^{13} := (m \otimes I)\circ \tau^{23}$.  
\begin{eqnarray*}
(I \otimes \Gamma) \circ \Delta (t) &=& (I \otimes \Gamma)(\sum_{c \in {\mop{\tiny{Adm}}(t)}} P^c(t)\otimes R^c(t))\\ 
&=&\sum_{c \in {\mop{\tiny{Adm}}(t)}} P^c(t)\otimes \Gamma (R^c(t))\\
&=& \sum_{c \in {\mop{\tiny{Adm}}(t)}\;,\; (t)} P^c(t)\otimes R^c(t^{(1)}) \otimes R^c(t^{(2)}).
\end{eqnarray*}
\begin{eqnarray*}
(m^{13} \otimes I) \circ (\Delta \otimes \Delta)\circ \Gamma(t) &=& (m^{13} \otimes I)\big( \sum_{(t)} \Delta(t^{(1)}) \otimes  \Delta (t^{(2)})\big) \\ 
&=& (m^{13} \otimes I) \big(\hspace{-1.5cm}\sum_{c_1 \in {\mop{\tiny{Adm}}(t^{(1)})}\;,\; c_2 \in {\mop{\tiny{Adm}}(t^{(2)})}}\hspace{-1.5cm} P^{c_1}(t^{(1)})\otimes R^{c_1}(t^{(1)}) \otimes P^{c_2}(t^{(2)}) \otimes R^{c_2}(t^{(2)})\big)\\
&=&\hspace{-0.2cm}\sum_{c_1 \in {\mop{\tiny{Adm}}(t^{(1)})}\;,\; c_2 \in {\mop{\tiny{Adm}}(t^{(2)})}} \hspace{-1cm} \big(P^{c_1}(t^{(1)}) P^{c_2}(t^{(2)})\big)\otimes R^{c_1}(t^{(1)}) \otimes R^{c_2}(t^{(2)}))\\
&=&\sum_{c \in {\mop{\tiny{Adm}}(t)}\;,\; (t)} P^{c}(t) \otimes R^{c}(t^{(1)}) \otimes R^{c}(t^{(2)}),
\end{eqnarray*}
which proves the theorem.
\end{proof}

\section{The doubling pre-Lie algebra of rooted trees}

\begin{definition} Let $(t_1 , s_1)$ and $(t_2, s_2)$ be two elements of $V$, we define the map $\leadsto$ by:
\begin{equation}
(t_1 , s_1)  \leadsto (t_2, s_2) := \sum_{v \in \Cal V(t_2 - s_2)}  (t_1 \rightarrow_v t_2 , s_1  s_2),
\end{equation}
where $v \in \Cal V(t_2 - s_2)$ denotes that $v$ is a vertex of $t_2$ but is not a vertex of $s_2$.
\end{definition} 
\begin{ex} 
\begin{eqnarray*}
(\echela, \racine)\leadsto (\arbrey, \racine) &=& (\arbrede ,\racine\racine) + (\arbredg, \racine\racine).
\end{eqnarray*}
\begin{eqnarray*}
(\arbrey, \echela)  \leadsto (\arbreha, \racine)   &=& (\arbremg, \racine\echela) +  (\arbremo, \racine\echela) + (\arbremi, \racine\echela) + (\arbremt, \racine\echela).
\end{eqnarray*}
\end{ex}
\begin{theorem} Equipped by $\leadsto$, the space $V$ is a pre-Lie algebra. 
\end{theorem}
\begin{proof} 
Let $(t_1 , s_1), (t_2 , s_2)$ and $(t_3 , s_3)$ be three elements of $V$, we have:
 \begin{eqnarray*}
(t_1 , s_1) \leadsto  \big[(t_2 , s_2) \leadsto (t_3 , s_3)\big]  &=& 
(t_1 , s_1) \leadsto  \big(\sum_{v \in \Cal V(t_3 - s_3)}  (t_2 \rightarrow_v t_3 , s_2 s_3)\big)\\
&=& \sum_{v \in \Cal V(t_3 - s_3)\; r \in \Cal V(t_2 \rightarrow_v t_3 -s_2 s_3)}  \big(t_1 \rightarrow_r (t_2 \rightarrow_v t_3) , s_1 s_2  s_3\big)\\
&=& \sum_{r\;,\;v \in \Cal V(t_3 - s_3)}  \big(t_1 \rightarrow_r (t_2 \rightarrow_v t_3) , s_1 s_2  s_3\big)\\
&& \;+\;   \sum_{v \in \Cal V(t_3 - s_3)\; r \in \Cal V(t_2 -s_2)}  \big(t_1 \rightarrow_r (t_2 \rightarrow_v t_3) , s_1 s_2 s_3\big).
\end{eqnarray*}
On the other hand we have: 
 \begin{eqnarray*}
\big[(t_1 , s_1) \leadsto (t_2 , s_2) \big]\leadsto (t_3 , s_3)  &=& \sum_{r \in \Cal V(t_2 - s_2)} (t_1 \rightarrow_r t_2 , s_1  s_2) \leadsto(t_3 , s_3)\\
&=& \sum_{v \in \Cal V(t_3 - s_3)\; r \in \Cal V(t_2 - s_2 )} \big((t_1 \rightarrow_r t_2) \rightarrow_v t_3 , s_1  s_2  s_3 \big) \\
&=& \sum_{v \in \Cal V(t_3 - s_3)\; r \in \Cal V(t_2 - s_2 )} \big((t_1 \rightarrow_r t_2) \rightarrow_v t_3 , s_1  s_2  s_3 \big).
\end{eqnarray*}
Then we have:
\begin{eqnarray*}
(t_1 , s_1) \leadsto  \big[(t_2 , s_2) \leadsto (t_3 , s_3)\big] &-& \big[(t_1 , s_1) \leadsto (t_2 , s_2) \big]\leadsto (t_3 , s_3)\\
  &=& \sum_{r\;,\;v \in \Cal V(t_3 - s_3)}  \big(t_1 \rightarrow_r (t_2 \rightarrow_v t_3) , s_1 s_2  s_3\big)\\
&& \;+\;   \sum_{v \in \Cal V(t_3 - s_3)\; r \in \Cal V(t_2 -s_2)}  \big(t_1 \rightarrow_r (t_2 \rightarrow_v t_3) , s_1 s_2 s_3\big)\\
	&& \;-\;  \sum_{v \in \Cal V(t_3 - s_3)\; r \in \Cal V(t_2 - s_2 )} \big((t_1 \rightarrow_r t_2) \rightarrow_v t_3 , s_1  s_2  s_3 \big)\\
	&=&  \sum_{r\;,\;v \in \Cal V(t_3 - s_3)}  \big(t_1 \rightarrow_r (t_2 \rightarrow_v t_3) , s_1 s_2  s_3\big),
	\end{eqnarray*}
	which is symmetric on $(t_1 , s_1)$ and $(t_2 , s_2)$. Then we obtain:
\begin{eqnarray*}
(t_1 , s_1) \leadsto  \big[(t_2 , s_2) \leadsto (t_3 , s_3)\big] &-& \big[(t_1 , s_1) \leadsto (t_2 , s_2) \big]\leadsto (t_3 , s_3)\\
	&=&\\ (t_2 , s_2) \leadsto \big[(t_1 , s_1) \leadsto(t_3 , s_3)\big] &-& \big[(t_2 , s_2) \leadsto (t_1 , s_1) \big]\leadsto(t_3 , s_3).
\end{eqnarray*}
Consequently, $\leadsto$ is pre-Lie. 
\end{proof}
\vspace{0.2cm}
\section{The enveloping algebra of the doubling pre-Lie algebra of rooted trees}
We showed that $(V, \leadsto)$ is a pre-Lie algebra, so we consider the Hopf symmetric algebra $\Cal D' : = \Cal S (V)$ equipped with its usual unshuffling coproduct $\chi$. We extend the product $\leadsto $ to $\Cal D'$ by using Definition \ref{odam} and we define a product $\bigstar$ on $\Cal D'$ by: 
\begin{eqnarray*}
(t, s) \bigstar (t', s') &=& \sum_{(t, s)}  (t, s)^{(1)} \big((t, s)^{(2)} \leadsto  (t', s')\big).
\end{eqnarray*}

By construction, the space $(\Cal D', \bigstar, \chi)$ is a Hopf algebra. 

\begin{theorem} $(\Cal D', m, \chi)$ is a comodule-coalgebra on $(\Cal D,  m, \Delta)$.
\end{theorem}
\begin{proof} 
It is clear taht $\Delta : \Cal D' \longrightarrow \Cal D  \otimes   \Cal D'$ is a coaction, that means that $\Delta$ is coassocitive. 

Second, we prove that the coproduct $\chi$ is a morphism of left $\Cal D'$-comodules. This amounts to the commutativity of the following diagram:
\diagramme{
\xymatrix{
  \Cal D' \ar[d]_{\chi} \ar[rr]^{\Delta} 
&& \Cal D \otimes  \Cal D' \ar[d]^{I \otimes \chi}\\
 \Cal D'  \otimes  \Cal D'  \ar[d]_{\Delta\otimes \Delta}&& \Cal D \otimes   \Cal D'  \otimes   \Cal D'  \\
  \Cal D  \otimes   \Cal D' \otimes  \Cal D \otimes   \Cal D'  \ar[rr]_{\tau^{23}} 
&&\ar[u]_{m \otimes I}   \Cal D  \otimes  \Cal D \otimes    \Cal D'  \otimes  \Cal D' 
 }
}

Let $(t, s) \in \Cal D'$, we have:
\begin{eqnarray*}
(I \otimes \chi) \circ \Delta (t, s) &=& (I \otimes \chi) \left(\sum_{c \in {\mop{\tiny{Adm}}(s)}} \big(t, P^c(s)\big)\otimes \big( R^c(t), R^c(s)\big) \right)\\ 
&=& \sum_{c \in {\mop{\tiny{Adm}}(s)}} \big(t, P^c(s)\big)\otimes \chi \big( R^c(t), R^c(s)\big)\\
&=& \sum_{c \in {\mop{\tiny{Adm}}(s)}} \big(t, P^c(s)\big)\otimes \big( R^c(t), R^c(s)\big)^{(1)}\otimes \big( R^c(t), R^c(s)\big)^{(2)}.
\end{eqnarray*}
\begin{eqnarray*}
(m^{13} \otimes I) \circ (\Delta \otimes \Delta)\circ \chi (t, s) &=& (m^{13} \otimes I)\left( \sum_{(t, s)} \Delta\big((t, s)^{(1)}\big) \otimes  \Delta \big((t, s)^{(2)}\big)\right) \\ 
&=& (m^{13} \otimes I)\left( \sum_{(t, s)} \Delta\big((t^{(1)}, s^{(1)})\big) \otimes \Delta \big((t^{(2)}, s^{(2)})\big)\right) \\ 
&=& (m^{13} \otimes I) \Big( \sum_{c' \in {\mop{\tiny{Adm}}(s^{(1)})}\;,\;c'' \in {\mop{\tiny{Adm}}(s^{(2)})}} (t^{(1)}, P^{c'}(s^{(1)}))\\
&&\hspace{0.2cm}\otimes \big( R^{c'}(t^{(1)}), R^{c'}(s^{(1)})\big)\otimes (t^{(2)}, P^{c''}(s^{(2)}))\otimes \big( R^{c''}(t^{(2)}), R^{c''}(s^{(2)})\big)\Big)\\
&=& \sum_{c' \in {\mop{\tiny{Adm}}(s^{(1)})}\;,\;c'' \in {\mop{\tiny{Adm}}(s^{(2)})}} \big(t^{(1)} t^{(2)}, P^{c'}(s^{(1)})P^{c''}(s^{(2)})\big)\\
&&\hspace{2cm} \otimes \big( R^{c'}(t^{(1)}), R^{c'}(s^{(1)})\big)\otimes \big( R^{c''}(t^{(2)}), R^{c''}(s^{(2)})\\
&=& \sum_{(t, s)\;,\;c \in {\mop{\tiny{Adm}}(s)}} (t, P^{c}(s)) \otimes \big( R^{c}(t^{(1)}), R^{c}(s^{(1)})\big)\otimes \big( R^{c}(t^{(2)}), R^{c}(s^{(2)})\\
&=& \sum_{c \in {\mop{\tiny{Adm}}(s)}} \big(t, P^c(s)\big)\otimes \big( R^c(t), R^c(s)\big)^{(1)}\otimes \big( R^c(t), R^c(s)\big)^{(2)},
\end{eqnarray*}
which proves the theorem.
\end{proof}
\section{Relations between the two pre-Lie structures}
In this section, we prove that $V$ is a left pre-Lie module on $\Cal T$ and we find some relations between the two pre-Lie structures defined on $V$ and $\Cal T$. Also we show that the enveloping algebra of the pre-Lie algebra $(\Cal D', \bigstar, \chi)$ is a module-bialgebra on $(\Cal H', \star, \Gamma)$.

\subsection{Left pre-Lie module}
\begin{definition}
Let $(\Cal A, \circ)$ be a pre-Lie algebra. A left $\Cal A$-module is a vector space $\Cal M$ provided with a bilinear law noted $\succ : \Cal A \otimes \Cal M \longrightarrow \Cal M$ such that for all $a, b \in \Cal A$ and $m \in \Cal M$, we have:
\begin{equation}
   a \succ (b \succ m) - (a \circ b)\succ m = b \succ (a \succ m) - (b \circ a) \succ m. 
\end{equation}
\end{definition}
\begin{definition} 
Let $t_1 \in {\Cal T}$ and $(t_2,  s_2) \in V$, we define the map $\Diamond$ by:
\begin{equation}
t_1 \Diamond (t_2, s_2) := \sum_{v \in \Cal V(s_2)}  (t_1 \rightarrow_v t_2 , t_1 \rightarrow_v s_2). 
\end{equation} 
\end{definition}
\begin{theorem} Equipped by $\Diamond$, the space $V$ is a left pre-Lie module on ${\Cal T}$. In other words for any $t_1, t_2 \in {\Cal T}$ and $(t_3 , t_3) \in V$, we have:
$$
t_1 \Diamond \Big[t_2 \Diamond (t_3 , s_3)\Big] - (t_1  \rightarrow t_2) \Diamond (t_3 , s_3) = t_2 \Diamond \Big[t_1 \Diamond(t_3 , s_3)\Big] - (t_2  \rightarrow t_1) \Diamond(t_3 , s_3).$$ 
\end{theorem}
\begin{proof} 
Let $t_1, t_2$ be two elements of ${\Cal T}$, and let $(t_3 , s_3)$ be an element of $V$, we have:
 \begin{eqnarray*}
t_1 \Diamond \Big[t_2 \Diamond (t_3 , s_3)\Big]  &=& 
t_1 \Diamond \Big[\sum_{r \in \Cal V(s_3)}  (t_2 \rightarrow_r t_3 , t_2 \rightarrow_r s_3)\Big]\\
&=& \sum_{r \in \Cal V(s_3)\; l \in \Cal V(t_2 \rightarrow_r s_3)}  [t_1 \rightarrow_l (t_2 \rightarrow_r t_3) , t_1 \rightarrow_l (t_2 \rightarrow_r s_3)]\\
&=& \sum_{l\;, \; r \in \Cal V(s_3)}  [t_1 \rightarrow_l (t_2 \rightarrow_r t_3) , t_1 \rightarrow_l (t_2 \rightarrow_r s_3)]   \\
&& \;+\; \sum_{l \in \Cal V(s_2)\;, \; r \in \Cal V(s_3)}  [t_1 \rightarrow_l (t_2 \rightarrow_r t_3) , t_1 \rightarrow_l (t_2 \rightarrow_r s_3)].
\end{eqnarray*}
On the other hand we have:
 \begin{eqnarray*}
(t_1  \rightarrow t_2) \Diamond(t_3 , s_3)  &=& 
\sum_{l \in \Cal V(t_2)}t_1 \rightarrow_l t_2 \Diamond(t_3 , s_3)\\
&=& 
\sum_{l \in \Cal V(t_2)\; r \in \Cal V(s_3)} \big[(t_1 \rightarrow_l t_2)\rightarrow_r t_3 , (t_1 \rightarrow_l t_2) \rightarrow_r s_3\big]\\
&=& 
\sum_{l \in \Cal V(t_2)\; r \in \Cal V(s_3)} \big[(t_1 \rightarrow_l t_2)\rightarrow_r t_3 , (t_1 \rightarrow_l t_2) \rightarrow_r s_3\big].
\end{eqnarray*}
Then, we have:
\begin{eqnarray*}
t_1 \Diamond \Big[t_2 \Diamond (t_3 , s_3)\Big] &-& (t_1  \rightarrow t_2) \Diamond (t_3 , s_3) \\ &=& \sum_{l\;, \; r \in \Cal V(s_3)}  [t_1 \rightarrow_l (t_2 \rightarrow_r t_3) , t_1 \rightarrow_l (t_2 \rightarrow_r s_3)]   \\
&& \;+\; \sum_{l \in \Cal V(s_2)\;, \; r \in \Cal V(s_3)}  [t_1 \rightarrow_l (t_2 \rightarrow_r t_3) , t_1 \rightarrow_l (t_2 \rightarrow_r s_3)]\\
&& \;-\;\sum_{l \in \Cal V(t_2)\; r \in \Cal V(s_3)} \big[(t_1 \rightarrow_l t_2)\rightarrow_r t_3 , (t_1 \rightarrow_l t_2) \rightarrow_r s_3\big]\\
&=&\sum_{l\;, \; r \in \Cal V(s_3)}  [t_1 \rightarrow_l (t_2 \rightarrow_r t_3) , t_1 \rightarrow_l (t_2 \rightarrow_r s_3)]  ,
\end{eqnarray*} 
which is symmetric in $t_1$ and $t_2$. Therefore:
$$
t_1 \Diamond \Big[t_2 \Diamond (t_3 , s_3)\Big] - (t_1  \rightarrow t_2) \Diamond (t_3 , s_3) = t_2 \Diamond \Big[t_1 \Diamond(t_3 , s_3)\Big] - (t_2  \rightarrow t_1) \Diamond(t_3 , s_3),$$
which proves the theorem.
\end{proof}
\subsection{Relation between $\leadsto$ and $\Diamond$}
In this subsection, we prove that there exist relations between the action $\Diamond$ and the pre-Lie product $\leadsto$ defined on $V$. 

\begin{theorem} The law $\Diamond$ is a derivation of the algebra $(V, \leadsto )$. In other words, for any $t_1 \in \Cal {T}$ and  $(t_2, s_2), (t_3, s_3) \in V$, we have:
\begin{equation*}
t_1 \Diamond  \big((t_2, s_2) \leadsto  (t_3, s_3) \big) = \big(t_1 \Diamond (t_2, s_2) \big) \leadsto  (t_3, s_3) + (t_2, s_2) \leadsto  \big(t_1\Diamond (t_3, s_3) \big).
\end{equation*}
\end{theorem}
\begin{proof}
Let $t_1 \in \wt {T}$ and  $(t_2, s_2), (t_3, s_3)\in V$, we have:
\begin{eqnarray*}
t_1 \Diamond  \big((t_2, s_2) \leadsto  (t_3, s_3) \big) &=& 
t_1 \Diamond \big(\sum_{v \in \Cal V(t_3 - s_3)} (t_2  \rightarrow_v t_3, s_2 s_3) \big)\\
&=&\sum_{r \in  \Cal V(s_2 s_3)\;\; v \in \Cal V(t_3 - s_3)} \big(t_1  \rightarrow_r(t_2  \rightarrow_v t_3) , t_1  \rightarrow_r(s_2 s_3) \big)\\ 
&=&\sum_{r \in  \Cal V(s_2)\;\; v \in \Cal V(t_3 - s_3)} \big((t_1  \rightarrow_r t_2)  \rightarrow_v t_3 , (t_1  \rightarrow_r s_2) s_3 \big)\\ 
&&\;\; +\;\; \sum_{r \in  \Cal V(s_3)\;\; v \in \Cal V(t_3 - s_3)} \big(t_2  \rightarrow_v (t_1  \rightarrow_r t_3) , s_2  (t_1\rightarrow_r s_3) \big)\\ 
&=&
\big(t_1 \Diamond (t_2, s_2) \big) \leadsto  (t_3, s_3) + (t_2, s_2) \leadsto  \big(t_1\Diamond (t_3, s_3) \big).
\end{eqnarray*}
\end{proof}

\begin{theorem}
The following diagram is commutative:
\diagramme{
\xymatrix{
\Cal {T} \otimes V \ar[d]_{I\otimes P_2} \ar[rrr]^{\Diamond} 
&&&V \ar[d]^{P_2}\\
\Cal {T}\otimes\Cal H\ar[rrr]_{\rightarrow} 
&&& \Cal H}
}
In other words, the projection on the second component $P_2$ is a morphism of pre-Lie modules.
\end{theorem}
\begin{proof}
Let be $t_1 \in \Cal {T}$ and $(t_2, s_2) \in   V$, we have:
\begin{eqnarray*}
P_2 \big(t_1 \Diamond (t_2, s_2)\big) &=& P_2 \big(\sum_{v \in \Cal V(s_2)} (t_1  \rightarrow_v t_2, t_1 \rightarrow_v  s_2) \big)\\
&=&  \sum_{v \in \Cal V(s_2)}  t_1  \rightarrow_v s_2\\
&=&  t_1 \rightarrow  s_2 \\
&=&  t_1 \rightarrow  P_2 (t_2, s_2)\\
&=&  (I \otimes P_2) \big(t_1 \otimes(t_2, s_2)\big),
\end{eqnarray*}
which proves the theorem.
\end{proof}

\subsection{Module-bialgebra}
\begin{proposition} 
Let $(t_1 , s_1)$ and $(t_2 , s_2)$ be two elements of $\Cal D'$. The product $\bigstar$ satisfies the following result:
$$(t_1 , s_1) \bigstar(t_2 , s_2) = 
(t_1 \star t_2 - t_1 \starg t_2 , s_1 s_2),$$
where:$$
t_1 \starg t_2 := \sum_{\substack{v \in \Cal V(s_2)\\{(t_1)}}} (t_1)^{(1)} (t_1)^{(2)}\rightarrow_v t_2.
$$
\end{proposition}
\begin{proof} 
\begin{eqnarray*}
(t_1 , s_1) \bigstar(t_2 , s_2) &=& \sum_{(t_1, s_1)}  (t_1, s_1)^{(1)} \big((t_1, s_1)^{(2)} \leadsto (t_2, s_2)\big)\\ 
&=& \sum_{\substack{v \in \Cal V(t_2 - s_2)\\(t_1, s_1) }} (t^{(1)} _1, s^{(1)} _1) (t^{(2)}_{1} \rightarrow_v t_2 , s^{(2)}_1  s_2)\\
&=& \sum_{\substack{v \in \Cal V(t_2 - s_2)\\(t_1, s_1) }}  \Big(t^{(1)}_1 (t^{(2)}_{1} \rightarrow_v t_2) , s^{(1)}_1 s^{(2)}_1  s_2 \Big)\\
&=& \sum_{\substack{v \in \Cal V(t_2 - s_2)\\(t_1)}}  \big(t^{(1)}_1(t^{(2)}_{1} \rightarrow_v t_2) , s_1 s_2\big)\\
&=& \Big(\sum_{\substack{v \in \Cal V(t_2)\;\;,\; (t_1)}}  t^{(1)}_1(t^{(2)}_{1} \rightarrow_v t_2) - \sum_{\substack{v \in \Cal V(s_2)\;\;,\; (t_1)}}   t^{(1)}_1(t^{(2)}_{1} \rightarrow_v t_2) , s_1 s_2\Big)\\
&=&(t_1 \star t_2 - t_1 \starg  t_2 , s_1 s_2).
\end{eqnarray*}
\end{proof}

\begin{theorem} $(\Cal D', \bigstar, \chi)$ is a module-bialgebra on $(\Cal H', \star, \Gamma)$.
\end{theorem}

\begin{proof}
We consider the map: $\alpha:   {\Cal D'} \otimes   {\Cal H'} \longrightarrow   {\Cal D'}$ defined for all $(t , s) \in   {\Cal D'}$ and $t' \in   {\Cal H'}$ by:
$$\alpha ((t , s)  \otimes t') = (t \star t', s).$$
To prove this theorem, first we will show that $\alpha$ is an action which results from the following commutative diagram:
\diagramme{
\xymatrix{
{\Cal D'}\otimes {\Cal H'} \otimes {\Cal H'}  \ar[d]_{I\otimes \star} \ar[rrr]^{\alpha \otimes I} 
&&& {\Cal D'}  \otimes  {\Cal H'} \ar[d]^{\alpha}\\
 {\Cal D'} \otimes  {\Cal H'} \ar[rrr]_{\alpha} 
&&&  {\Cal D'} }
}
Let $(t , s) \in  {\Cal D'}$ and $t_1, t_2 \in   {\Cal H'}$, we have:
\begin{eqnarray*}
\alpha \circ (\alpha \otimes I) [(t, s)\otimes t_1 \otimes t_2]  &=& \alpha [(t \star t_1, s) \otimes t_2]\\ 
&=& \big((t \star t_1) \star t_2, s\big) \\
&=& \big(t \star (t_1 \star t_2), s\big)\\
&=& \alpha [(t , s) \otimes t_1 \star t_2]\\
&=& \alpha \circ (I \otimes \star) [(t , s) \otimes t_1 \otimes t_2].
\end{eqnarray*}
Second, we show that the following diagram is commutative:
\diagramme{
\xymatrix{
   {\Cal D'} \otimes   {\Cal D'}  \otimes    {\Cal H'} \ar[d]_{I\otimes I \otimes \Gamma} \ar[rr]^{\bigstar \otimes I} 
&&  {\Cal D'} \otimes    {\Cal H'}  \ar[d]^{\alpha}\\
    {\Cal D'}  \otimes    {\Cal D'}   \otimes    {\Cal H'} \otimes    {\Cal H'} \ar[d]_{\tau^{23}}&&   {\Cal D'} \\
  {\Cal D'} \otimes   {\Cal H'} \otimes   {\Cal D'} \otimes   {\Cal H'}\ar[rr]_{\alpha \otimes \alpha} 
&&\ar[u]_{\bigstar}  {\Cal D'} \otimes    {\Cal D'}
 }
}
We use the following notation: $\Gamma^{23} = \tau^{23} \circ (I\otimes I \otimes \Gamma).$
\begin{eqnarray*}
\bigstar \circ (\alpha \otimes \alpha) \circ \Gamma^{23}\big((t_1, s_1)\otimes (t_2, s_2)\otimes t\big)  &=&\sum_{(t)} \alpha \big((t_1 , s_1)\otimes t^{(1)}\big) \bigstar \alpha \big((t_2 , s_2)\otimes t^{(2)}\big)\\ 
&=& \sum_{(t)} (t_1 \star t^{(1)}, s_1) \bigstar (t_2 \star t^{(2)}, s_2)\\ 
&=& \sum_{\substack{(t)}} \big((t_1 \star t^{(1)})\star (t_2 \star t^{(2)}) - (t_1 \star t^{(1)})\starg (t_2 \star t^{(2)}), s_1 s_2\big)\\
&=& \big((t_1 \star t_2)\star t - (t_1 \starg t_2)\star t, s_1 s_2\big).
\end{eqnarray*}
On the other hand:
\begin{eqnarray*}
\alpha \circ (\bigstar \otimes I) \big((t_1, s_1)\otimes (t_2, s_2)\otimes t\big)  &=& \alpha \big((t_1, s_1)\bigstar (t_2, s_2)\otimes t\big)\\ 
&=& \alpha \big((t_1 \star t_2 - t_1 \star t_2 , s_1  s_2)\otimes t\big)\\ 
&=& \big((t_1 \star t_2 - t_1 \starg t_2) \star t, s_1 s_2\big)\\ 
&=& \big((t_1 \star t_2)\star t - (t_1 \starg t_2) \star t, s_1 s_2\big).
\end{eqnarray*}
Finally, we prove that the coproduct $\chi$ is a morphism of modules. This amounts to the commutativity of the following diagram:
\diagramme{
\xymatrix{
   {\Cal D'}  \otimes  {\Cal H'}  \ar[d]_{I\otimes \Gamma} \ar[rr]^{\alpha} 
&&  {\Cal D'} \ar[d]^{\chi}\\
   {\Cal D'}  \otimes    {\Cal H'}  \otimes   {\Cal H'}  \ar[d]_{\chi \otimes I \otimes I}&&  {\Cal D'}  \otimes   {\Cal D'} \\
  {\Cal D'} \otimes   {\Cal D'}  \otimes   {\Cal H'}  \otimes   {\Cal H'}  \ar[rr]_{\tau^{23}} 
&&\ar[u]_{\alpha \otimes \alpha}   {\Cal D'}  \otimes   {\Cal H'}  \otimes   {\Cal D'}  \otimes   {\Cal H'} 
 }
}
Let $(t_1 , s_1) \in   {\Cal D'}$ and $t_2 \in {\Cal H'}$, we have:
\begin{eqnarray*}
\chi \circ \alpha \big((t_1, s_1)\otimes t_2\big) &=& \chi \big(t_1 \star t_2, s_1 \big)\\ 
&=& \sum_{\substack{(t_1)}}  \chi  \big(t^{(1)}_1 (t^{(2)}_1 \triangleright t_2) , s\big)  \\
&=& \sum_{\substack{(t_1, s_1)\;,\; (s_2)}} (t^{(11)}_1 (t^{(12)}_1 \triangleright t^{(1)}_2) , s^{(1)})   \otimes (t^{(21)}_1 (t^{(22)}_1 \triangleright   t^{(2)}_2) , s^{(2)})\\
&=& \sum_{\substack{(t_1, s_1)\;,\; (t_2)}} (t^{(1)}_1 \star  t^{(1)}_2 , s^{(1)})   \otimes (t^{(2)}_1 \star t^{(2)}_2 , s^{(2)}).
\end{eqnarray*}
We use the notation: $(\chi \otimes I \otimes I) \circ (I \otimes \Gamma) = \chi \otimes \Gamma$. 
\begin{eqnarray*}
(\alpha\otimes \alpha) \circ \tau^{23}\circ (\chi \otimes \Gamma) \big((t_1, s_1)\otimes t_2\big)
&=& (\alpha\otimes \alpha) \big(\sum_{\substack{(t_1, s_1)\;,\; (t_2)}} (t^{(1)}_1, s^{(1)}_1)\otimes  t^{(1)}_2 \otimes (t^{(2)}_1, s^{(2)}_1) \otimes  t^{(2)}_2\big)\\ 
 &=& \sum_{\substack{(t_1, s_1)\;,\; (t_2)}} \alpha\big((t^{(1)}_1, s^{(1)}_1)\otimes  t^{(1)}_2 \big)\otimes \alpha\big((t^{(2)}_1, s^{(2)}_1) \otimes  t^{(2)}_2\big)\\ 
&=& \sum_{\substack{(t_1, s_1)\;,\; (t_2)}} (t^{(1)}_1 \star  t^{(1)}_2 , s^{(1)})   \otimes (t^{(2)}_1 \star t^{(2)}_2 , s^{(2)}),
\end{eqnarray*}
which proves the theorem.
\end{proof}


\end{document}